\documentclass[12pt]{amsart}
\usepackage[margin=1.25in]{geometry}
\usepackage{mathpack}
\usepackage{hyperref}
\usepackage{bm}
\usepackage{float}
\usepackage{graphicx}

\newcommand{\ce}[1]{\ceil{#1}}
\newcommand{\fl}[1]{\floor{#1}}

\renewcommand{\hat}{\widehat}
\usepackage{mathtools}
\DeclarePairedDelimiter{\ceil}{\lceil}{\rceil}

\DeclareMathOperator{\divisor}{div}

\DeclareMathOperator{\Ann}{Ann}

\DeclareMathOperator{\Spec}{Spec}
\DeclareMathOperator{\Int}{Int}
\DeclareMathOperator{\Newt}{Newt}
\DeclareMathOperator{\Hom}{Hom}

\usepackage{wrapfig}

\usetikzlibrary{calc}

\newcommand{\ms}[1]{\mathscr{#1}}

\title{Big Cohen-Macaulay Test Ideals on Mixed Characteristic Toric Schemes }
\author{Marcus Robinson}
\begin{document}
\maketitle



\begin{abstract}
We provide a formula to compute the big Cohen-Macaulay test ideal for triples $((R,\Delta),\mf{a}^{t})$ where $R$ is a mixed characteristic toric ring and $\mf{a}$ is a monomial ideal. Of particular interest is that this result is consistent with the formulas for test ideals in positive characteristic and multiplier ideals in characteristic zero. 
\end{abstract}

\section{Introduction}

In his proof of the direct summand conjecture in mixed characteristic, Andr\'e used the construction of a large integral perfectoid extension where an interesting element of the base ring has a set of compatible $p$th roots \cite{AND1}. Using the subsequent innovation of weakly functorial integral perfectoid big Cohen-Macaualy $R^{+}$-algebras \cite{AND3} Ma and Schwede define a mixed characteristic test/multiplier ideal object \cite{MS2} called the big Cohen-Macaulay test ideal.  

Let $R$ be an equal characteristic regular domain satisfying an additional geometric assumption (essentially of finite type, complete, F-finite if characteristic $p>0$). Let $\mf{a} \subseteq R$ be an ideal and $t \geq 0$. If $R$ is a $d$-dimensional local ring with maximal ideal $\mf{m}$ and $\pi:Y \rightarrow X = \Spec R$ is a log resolution of singularities of $(R, \mf{a})$ where $\mf{a} \cdot \mc{O}_{Y} = \mc{O}(-A)$, then by Matlis duality we can define the multiplier ideal
$$
\mc{J}(R, \mf{a}^{t}) = \Ann_{R}\left( \left(\ker H^{d}_{\mf{m}}(R) \rightarrow \mathbb{H}^{d}_{\mf{m}}(\bold{R}\pi_{*}\mc{O}_{Y}(\lfloor tA \rfloor))\right)\right). 
$$
In characteristic $p>0$ we define the test ideal by replacing the log resolution of singularities $Y \rightarrow \Spec R$ with $R \subseteq R^{1/p^{\infty}}$ up to perturbations by $c^{1/p^{k}}$ with $k \gg 0$ for some fixed $c$ called a test element. In particular
$$
\tau(R, \mf{a}^{t}) = \Ann_{R}\{\eta \in H^{d}_{\mf{m}} \ | \  0 = c^{1/p^{e}}(\mf{a}^{\lceil tp^{e} \rceil})^{1/p^{e}} \eta \in H^{d}_{\mf{m}}(R^{1/p^{e}}) \mbox{ for all } e \gg 0 \}
$$.
Using big Cohen-Macaulay $R^{+}$-algebras as the appropriate substitution for resolution of singularities, the big Cohen-Macaulay test deal with respect to $B$ is approximately defined as 
$$
\tau_{B}(R, \mf{a}^{t}) =  \Ann_{R} (( \ker H^{d}_{\mf{m}}(R) \xrightarrow{(\mf{a}^{\ceil{tn}})^{1/n}} H^{d}_{\mf{m}}(B) )). 
$$
Big Cohen-Macaualay test ideals have proved useful in many of the same applications that their equal characteristic analogous have including symbolic powers \cite{MS1} and singularities \cite{MS2}. 

Due to the complicated nature of the big Cohen-Macaulay test ideals it is difficult to compute even simple examples.  Our main result computes the big Cohen-Macaulay test ideal of a monomial ideal on a toric scheme over a ring of mixed characteristic.  

\begin{theorem}
Let $R = W[\sigma^{\vee} \cap M]$, $X = \Spec R$ a mixed characteristic toric scheme, $W$ a mixed characteristic DVR, $\mf{a}$ a monomial idea not containing $p$ and $B$ a sufficiently large big Cohen-Macaulay $R^{+}$-algebra. Let $(X, \Delta)$ be a pair with $\Delta$ a torus invariant $\mathbb{Q}$-divisor such that $K_{X} + \Delta$ is $\mathbb{Q}$-Cartier and torus invariant. Then there exists a monomial $x^{u}$  such that $\divisor x^{u}= r(K_{X} + \Delta)$ for some integer $r$. Setting $w = \frac{u}{r} $ we have that,
$$
\tau_{B}((R, \Delta), \mf{a}^{t}) = \{ x^{v} \in R \ | \ v-w  \in\Int(t\Newt(\mf{a}) \}. 
$$
\end{theorem}

\noindent This result is consistent with the multiplier ideal in characteristic zero and the traditional test ideal in characteristic $p>0$. \\

\noindent \textbf{Acknowledgments:} I would like to thank my advisor Karl Schwede for his insight and guidance. I would also like to thank Joaquin Moraga and Daniel Smolkin for valuable discussion. \\

\section{Background}

In this section review the necessary theory of Big Cohen-Macaulay test ideals and toric geometry. Throughout this paper $R$ will be a commutative ring with unity. For a local ring $(R,\mf{m})$ we will denote the absolute integral closure of $R$ inside of an algebraic closure of its fraction field by $R^{+}$. To define the big Cohen-Macaulay test ideal we will use the language of integral perfecoid algebras following the notation and conventions of \cite{MS2}.

\subsection{Big Cohen-Macaulay test ideals} 

In this section we will review the theory of big Cohen-Macaulay test ideals of pairs $(R, \Delta)$ developed in \cite{MS2}, a more complete exposition can also be found there. In section 3 we will extend these ideas into the settings of triples. For convenience we will often refer to the following setting. 

\begin{setting} \label{set} Let $(R, \mathfrak{m})$ be a complete normal local domain.
\begin{itemize}
    \item The divisor $\Gamma \geq 0$ will denote an effective $\mathbb{Q}$-Cartier divisor on $\Spec R$.
    \item The divisor $\Delta \geq 0$ will denote an effective $\mathbb{Q}$-divisor on $\Spec R$ such that $K_{R} + \Delta$ is $\mathbb{Q}$-Cartier. 
    \item If $\Delta \geq 0$ is defined, we fix an embedding $R \subseteq \omega_{R} \subseteq K(R)$ (hence we fix a canonical divisor $K_{R}$) such that $K_{R} + \Delta$ is effective. Since we are working locally, such a canonical divisor always exists. 
\end{itemize}
\end{setting}

\begin{defn}
Fix a big Cohen-Macaulay $R^{+}$-algebra B. With notation as above,  define
$$
0_{H^{d}_{\mf{m}}(R)}^{B,\Gamma} =  \ker \left( H^{d}_\mf{m}(R) \xrightarrow{f^{1/n}}H^{d}_\mf{m}(B) \right).
$$
where since $\Gamma$ is $\mathbb{Q}$-Cartier and effective there is some $f\in R$ such that $\divisor f = n\Gamma$ for some positive integer $n$. Moreover, if $R$ has mixed characteristic $(0,p)$, we define
$$
0_{H^{d}_{\mf{m}}(R)}^{\ms{B}, \Gamma} =\{ \eta \in 0_{H^{d}_{\mf{m}}(A)}^{B, \Gamma} \ | \ \mbox{$B$ an integral perfectoid big Cohen-Macaulay $R^{+}$-algebra} \} 
$$
We then define 
$$
\tau_{B}(\omega_{R}, \Gamma) = \Ann_{\omega_{R}}\left(0_{H^{d}_{\mf{m}}(R)}^{ B, \Gamma} \right),
$$
and in mixed characteristic
$$
\tau_{\mathscr{B}}(\omega_{R}, \Gamma) = \Ann_{\omega_{R}}\left(0_{H^{d}_{\mf{m}}(R)}^{ B, \Gamma} \right).
$$
We call $\tau_{B}(\omega_{R}, \Gamma)$ the BCM parameter test submodule of $(\omega_{R}, \Gamma)$ with respect to $B$ and $\tau_{\mathscr{B}}(\omega_{R}, \Gamma)$ the perfectoid BCM parameter test submodule of $(\omega_{R}, \Gamma)$
\end{defn}

\begin{remark} We record the following facts
\begin{itemize}
    \item The definition is independent of the choice of $f^{1/n}$. By working in an $R^{+}$-algebra we ensure that any choice of $f^{1/n}$ differs from another by a unit and so does not impact $0_{H^{d}_{\mf{m}}(R)}^{B,\Gamma}$ or $0_{H^{d}_{\mf{m}}(R)}^{\ms{B}, \Gamma}$.
    \item $\tau_{B}(\omega_{R}, K_{R} + \Delta)$ and $\tau_{\mathscr{B}}(\omega_{R}, K_{R} + \Delta)$ are independent of the choice of $K_{R}$ (\cite{MS2},Lemma 6.7)
    \item $\tau_{B}(\omega_{R}, K_{R} + \Delta)$ and $\tau_{\mathscr{B}}(\omega_{R}, K_{R} + \Delta)$ are ideals of $R$ (\cite{MS2}, Lemma 6.8).
\end{itemize}
\end{remark}

\begin{defn}
Based on the above remarks we define
$$
\tau_{B}(R, \Delta) = \tau_{B}(\omega_{R}, K_{R} + \Delta)
$$
and in mixed characteristic
$$
\tau_{\mathscr{B}}(R, \Delta) = \tau_{\mathscr{B}}(\omega_{R}, K_{R} + \Delta).
$$
We call $\tau_{B}(R, \Delta)$ the BCM test ideal of the pair $(R,\Delta)$ and $\tau_{\mathscr{B}}(R, \Delta)$ the perfectoid BCM test ideal of the pair $(R,\Delta)$. 
\end{defn}

\subsection{Toric setup} We record some notation and facts about mixed characteristic toric. For justification and references in the equal characteristic case \cite{FUL}. Fix a dual pair of lattices $N = M^{\vee} \cong \mathbb{Z}^{d}$.  Let 
$$\sigma = \{r_{1}u_{1} + \cdots + r_{t}u_{t} \ | \ r_{i} \in \mathbb{R}_{+}, u_{i} \in N \} \subset N_{\mathbb{R}} = N \otimes_{\mathbb{Z}} \mathbb{R}.$$
be a rational polyhedral cone. Throughout we will assume that $\sigma$ is strongly convex which means that it contains no positive dimensional subspace of $N_{\mathbb{R}}$. 

Let $\langle \cdot, \cdot \rangle$ be the pairing of the dual lattices $M$ and $N$. Then the dual cone, denoted  $\sigma^{\vee}$, is defined as 
$$
\{m \in M_{\mathbb{R}} \ | \ \langle m,v \rangle \geq 0 \mbox{ for all } v \in \sigma \}.
$$
 The lattice points of $\sigma^{\vee}$ describe a sub-semigroup of Laurent polynomials $k[x_{1}^{\pm 1 }, \ldots, x_{n}^{\pm 1}]$ by associating lattice points $\lambda = (\lambda_{1}, \ldots, \lambda_{n}) \in \sigma^{\vee}$ to monomials $x^{\lambda} = x_{1}^{\lambda_{1}} \cdots x_{n}^{\lambda_{n}}$. This identification defines a semigroup ring $R = W[\sigma^{\vee} \cap M]$ where $W$ is a mixed characteristic DVR such that $W/p$ is an algebraically closed field. When $\sigma^{\vee}$ is strongly convex, $R$ is normal and Cohen-Macaulay.  We will denote the associated mixed characteristic toric scheme as $X = \Spec R$.

Throughout we will study monomial ideals $\mf{a}$. An important geometric object connected to a monomial ideal in a toric ring is the Newton polyhedron denoted Newt$(\mf{a}) \subset M_{\mathbb{R}}$. The Newon polyhedron is defined as the convex hull of the set of lattice points $\lambda$ associated to the monomials $x^{\lambda}$ that are contained in $\mf{a}$.

Let $v_{1}, \ldots, v_{n}$ be the first lattice points on the edges of the cone $\sigma$. The orthogonals $v_{i}^{\perp} \cap \sigma^{\vee}$ define facets of $\sigma^{\vee}$, hence they define codimension one subschemes $D_{i}$ of $X$. In equal characteristic a divisor is called torus invariant if it is of the form $\sum a_{i}D_{i}$ with $a_{i} \in \mathbb{Z}$, for convenience we will keep this name in mixed characteristic even though there is no torus action. A divisor is $\mathbb{Q}$-Cartier if $D = \sum (w, v_{i}) D_{i}$ for some $w \in M_{\mathbb{Q}}$. There is a canonical choice of canonical divisor, namely $K_{X} = -\sum D_{i}$. 

\begin{lemma}
With $D_{i}$ defined as above, $-\sum D_{i}$ is a canonical divisor on $X = \Spec R = \Spec W[\sigma^{\vee} \cap M]$. 
\end{lemma}

\begin{proof}
Let $K_{X}$ be a canonical divisor on $X$ and $K_{X/W}$ the relative canonical of $X$ over $W$. Since $\Spec W$ is Gorenstein $K_{X/W} \sim K_{X}$ and we may assume $K_{X}$ is horizontal. Any two horizontal divisors are linearly equivalent if and only if they are linearly equivalent after inverting $p$. Then since $K_{X} \otimes R_{p} = K_{R_{p}}$ and $-\sum D_{i}\otimes R_{p} = K_{R_{p}}$ it follows that $-\sum D_{i}$ is a canonical divisor on $X$.  
\end{proof}

\subsection{Multiplier ideals and test ideals of monomial ideals} 

We briefly recall the definition of the multiplier ideal in characteristic zero and the test ideal in positive characteristic. In characteristic 0, let $(X, \Delta)$, be a pair consisting of a normal variety $X$ and a $\mathbb{Q}$-divisor $\Delta$ such that $K_{X} + \Delta$ is $\mathbb{Q}$-Cartier. Given an ideal sheaf $\mf{a}$ on $X$, take a log resolution $\mu: Y \rightarrow X$ of $\mf{a}$ that is simultaneously a log resolution of the pair. Then for all rational $t >0$ we define the multiplier ideal of $\mf{a}^{t}$ on the pair $(X, \Delta)$ to be
$$
\mathcal{J}((X, \Delta), \mf{a}^{t}) = \mu_{*}\mathcal{O}_{Y}(K_{Y} - \floor{\mu^{*}(K_{X} + \Delta) + tA})
$$
where  $A$ ia an effective normal crossing divisor given by $\mf{a}\cdot \mathcal{O}_{Y} = \mathcal{O}_{Y}(-A)$.

If $(R,\mf{m})$ is instead a local ring of characteristic $p>0$ and $\mf{a}$ is an ideal then the test ideal is defined as:
$$
\tau(R, \mf{a}^{t}) = \Ann_{R}\{\eta \in H^{d}_{\mf{m}}(R) \ | \  c^{1/p^{e}}(\mf{a}^{\ceil{tp^{e}}})^{1/p^{e}} \eta = 0 \in H^{d}_{\mf{m}}(R^{1/p^{\infty}}) \mbox{ for all } e \gg 0\}
$$

In either setting it was shown that multiplier/test ideal of a monomial ideal of a toric variety can be computed using a remarkably simple formula. 

\begin{theorem} [\cite{B2004}, \cite{H2001}, \cite {HY2003}] \label{BL}
Let $X = \Spec R$ be a toric variety over a field of characteristic zero. Let $(X, \Delta)$ be a pair with $\Delta$ a torus invariant $\mathbb{Q}$-divisor such that $K_{X} + \Delta$ is $\mathbb{Q}$-Cartier. Since $K_{X} + \Delta$ is $\mathbb{Q}$-Cartier and torus invariant, there exists a monomial $x^{u}$  such that div $x^{u}= r(K_{X} + \Delta)$ for some integer $r$. Setting $w = \frac{u}{r} $ we have that,
$$
\mathcal{J}((X, \Delta), \mf{a}^{t}) = \{ x^{v} \in R \ | \ v-w  \in\Int(t\Newt(\mf{a}) \}. 
$$
If $X$ is instead a toric variety over a field of positive characteristic then
$$
\tau(\mf{a}^{t}) = \{ x^{\lambda} \in A \ | \ \exists \ \beta \mbox{ with } \langle \beta, v_{i} \rangle \leq 1 \mbox{ for all $i$, such that }\ \lambda + \beta  \in \Int(t\Newt(\mf{a})) \}. 
$$
where the  $v_{i}$ are the first lattice points of the edges of $\sigma$. If in addition $R$ is $\mathbb{Q}$-Gorenstein, then there is a $w_{o}$  with $(w_{0}, v_{i}) = 1$ for all $i$. Therefore $x^{\lambda} \in \tau(\mf{a}^{t})$ if and only if $m + w_{0} \in$ interior of $c \Newt(\mf{a})$.
\end{theorem}

\noindent The above result for multiplier ideals was proved by Blickle \cite{B2004} as an extension of Howald's formula for monomial ideals \cite{H2001}. In the positive characteristic setting Hara and Yoshida \cite{HY2003} proved the $\mathbb{Q}$-Gorenstein case \cite{HY2003}. Their result was extended to the non  $\mathbb{Q}$-Gorenstein case by Blickle \cite{B2004}.

\section{Results}

In the first section we will give a few preliminary results necessary for the main result of section two - the proof of the formula for the BCM test ideal of pairs. In the third section we define the BCM test ideal of triples and extend our results to that setting. 

\subsection{Preliminaries}

 We will describe our mixed characteristic toric scheme using the language outlined in section 2. In particular we will work in the following setting: $X = \Spec R = \Spec W[\sigma^{\vee} \cap M]$, where $\sigma$ is a strongly convex rational polyhedral cone, $v_{1}, \ldots, v_{s}$ denote the first lattice points of the rays of $\sigma$ and $W$ denotes a mixed characteristic DVR such that $W/\varpi$ is algebraically closed. For convenience we will denote by $R_{n} = W[\sigma^{\vee} \cap \frac{1} {n}  M] $ where $n$ is a positive integer.  The ideal $\mf{n} \subset R$ will denote the ideal generated by $\{ x^{v}\ | \ v \in \sigma^{\vee} \cap M \setminus \{0\}\}$. Since $\sigma$ (and therefore $\sigma^{\vee}$) is strongly convex the ideal $\mf{m} = (\varpi,\mf{n})$ is a maximal ideal of $R$, $R$ is normal and Cohen-Macaulay.

One obstruction to computing the big Cohen-Macaulay test ideal is understanding the actual big Cohen-Macaulay $R^{+}$-algebra extension $B$ and its top local cohomology module $H^{d}_{\mf{m}}(B)$. In the toric case we are able to circumvent this issue by factoring the map $H^{d}_{\mf{m}}(R) \rightarrow H^{d}_{\mf{m}}(B)$ through $H^{d}_{\mf{m}}(R_{n})$ and showing that $H^{d}_{\mf{m}}(R_{n}) \hookrightarrow H^{d}_{\mf{m}}(B)$ is injective. 

\begin{defn}
A pair $\tau_{B}(R,\Delta)$ is $BCM_{B}$-regular if 
$$
\tau_{B}(R,\Delta) = R.
$$
\end{defn}

\begin{lemma} \label{pure}
Let $R_{n}$ be the mixed characteristic toric ring described at the start of the section and $B$ be a big Cohen-Macaulay $R^{+}$-algebra. Then the extension  $R_{n} \rightarrow B$ is pure. As a consequence $H^{d}_{\mf{m}}(R_{n}) \hookrightarrow H^{d}_{\mf{m}}(B)$ is injective for any $n\in \mathbb{Z}_{+}$. 
\end{lemma}

\begin{proof}
Fix $n$ and set $R = R_{n}$. By \cite{MS2} Theorem 6.12, if a pair $(R,\Delta)$ is BCM$_{B}$-regular then $R\rightarrow B$ is pure. By \cite{MS2} Corollary 6.29 a pair is BCM$_{B}$-regular if $K_{R} + \Delta$ is $\mathbb{Q}$-Cartier of index not divisible by $p$ and $(R/p, \Delta|_{R/p})$ is strongly $F$-regular. 

By Theorem \ref{BL}, $\tau(R/p, \Delta|_{R/p}) = R/p$ exactly when $\Delta|_{R/p} +K_{X} = \sum a_{i}D_{i}|_{R/p}$ where $ a_{i} < 0$. Pick a monomial $x^{v}$ with $v \in  \Int(\sigma^{\vee}) \cap M$. Then $\divisor(x^{\lambda}) = \sum a_{i} D_{i}$ with $a_{i} >0$ since $a_{i} = \langle v,v_{i}\rangle > 0$ because $v$ is in the interior of $\sigma^{\vee}$. Set $\Delta = \sum (1 -a_{i}/m) D_{i}$ with $m$ such that $a_{i}/m < 1$ for all $i$ and $p \not| m$. Then $\Delta+K_{x} = -\sum \frac{a_{i}}{ m }D_{i}$ is $\mathbb{Q}$-Cartier with index $m$ and the result follows.  
\end{proof}

Central to our proof will be understanding the local cohomology modules $H^{d}_{\mf{m}}(R_{n})$. 

\begin{lemma}\label{lcm}
Let $R_{n}, \mf{n}$ and $\mf{m}$ be as defined at the start of the section. Then $H^{d}_{\mf{m}}(A_{n})$ is spanned by monomials of the form $x^{-\lambda}/p^{a}$ where $\lambda \in \Int(\sigma^{\vee}) \cap \frac{1} {n} M$.
\end{lemma}

\begin{proof}
By \cite{24HR} Theorem 20.25, $H^{i}_\mf{n}(A_{n})$ can be computed using the Ishida complex
\begin{flushleft}
$\displaystyle 0 \rightarrow W[\sigma^{\vee} \cap \frac{1}{n}M] \rightarrow \bigoplus_{\mbox{rays }F} W[\sigma^{\vee} \cap \frac{1}{n}M]_{F} \rightarrow \cdots \bigoplus_{i\mbox{-faces }F} W[\sigma^{\vee} \cap \frac{1}{n}M]_{F}\xrightarrow{\delta^{i}}$
\end{flushleft}
\begin{flushright}
$
\xrightarrow{\delta^{i}} \cdots \rightarrow \bigoplus_{\mbox{facets }F} W[\sigma^{\vee} \cap \frac{1}{n}M]_{F} \xrightarrow{\delta^{d-1}} W[\mathbb{Z}^{d-1}] \rightarrow 0
$
\end{flushright}
where $W[\sigma^{\vee} \cap \frac{1}{n}M]_{F}$ is the semigroup ring with monomials of the form
$$
\{ x^{q-f}\ | \ q \in \sigma^{\vee} \cap \frac{1}{n}M \mbox{ and } f \in F \},
$$
where $F$ consists of the lattice points of $1/n M$ that intersect a fixed $i$-face of the cone $\sigma^{\vee}$. Note that the complex in \cite{24HR} is over a field but the computation is independent of the base ring. 

Since $R$ is Cohen-Macaulay the depth of $(p,\mf{n})$ is equal to the dimension $d$. Thus $\mf{n}$ has depth $n-1$ since $p$ is a regular element. It follows that $H^{i}_{\mf{n}}(R) = 0$ for any $i < d-1$. 

To compute $H^{d}_{\mf{m}}(R)$ observe that 
$$
\image \delta^{d-1} = \bigcup_{\substack{f \in F\\ F \mbox{ facet of } \sigma^{\vee}}} (\sigma^{\vee} - f) \cap \frac{1}{n} M
$$
then it is straightforward to verify that since $\sigma^{\vee}$ is pointed by hypothesis, the only lattice points of $\mathbb{Z}^{d-1}$ not in $\image \delta^{d-1}$ are $x^{-\lambda}$ where $\lambda \in \Int(\sigma^{\vee}) \cap \frac{1} {n} M$.

We may compute $H^{d}_{\mf{m}}(R)$ as a composition of the functors $H^{i}_{(\varpi)}(\underline{\ \ }) \circ H^{j}_{\mf{n}}(\underline{\ \ })$. Indeed, $H^{j}_{\mf{n}}(R)$ is nonzero only when $j= d-1$ via the Grothendiek spectral sequence $E^{i,j}_{2} = H^{i}_{(\varpi)}(H^{j}_{\mf{n}}(R)) \Rightarrow H^{i+j}_{(\varpi,n)}(R)$. This gives that   $H^{d}_{\mf{m}}(R) =H^{1}_{(\varpi)}\left(H^{d-1}_{\mf{n}}(R) \right)$. Then these are precisely monomials of the form $x^{-\lambda}/\varpi \in H^{d}_{\mf{m}}(A_{n})$ where $\lambda \in \Int(\sigma^{\vee}) \cap \frac{1} {n} M$.
\end{proof}

\begin{lemma} \label{inj}
Let $R = W[\sigma^{\vee} \cap M]$. The $R$-module $E$ spanned by monomials $\{ax^{v} \ | \ v \in -\sigma^{\vee} \cap M, a \in W_{p}/W \}$ is an injective hull of $R$
\end{lemma}

\begin{proof}
First, we show that $E$ is injective. By Baer's Criterion it suffices to show that for any ideal $\mf{a}\subseteq R$ and $R$-module homomorphism $\mf{a} \rightarrow E$ lifts to a map $R \rightarrow E$. This is clear since any $R$-module homomorphism must be an isomorphism on the underlying lattices defining $R$ and $E$.  

Next we show that $E$ is an essential extension of $k$. Take any proper submodule $N \subseteq E$ with $N \neq 0$. Then $N$ consists of $W_{p}/W$-linear combinations of monomials. Multiplying any element of $N$ by a suitable monomial with coefficient in $W$ yields a nonzero element of $k$. Thus $N \cap k \neq 0$ and so $E$ is an essential extension of $k$.  
\end{proof}

\subsection{BCM Test Ideal of Pairs}  We are now prepared to prove the main result. Let $X = \Spec R = \Spec W[\sigma^{\vee} \cap M]$ be a mixed characteristic toric scheme,  $\Delta$ an effective $\mathbb{Q}$-Cartier torus invariant divisor such that on $\Spec R$ such that $\Delta + K_{R}$ is $\mathbb{Q}$-Cartier. There is a monomial $x^{u}$ such that $\divisor x^{u} = r(K_{x} + \Delta)$ for some integer $r$. Set $w = u/r$. 

\begin{theorem} \label{MRG} 
Let $R$ and $\mf{m}$ be defined as at the start of the section, $\hat{R}$ be the $\mf{m}$-adic completion of $R$ at the maximal ideal $\mf{m}$ and $B$ be a big Cohen-Macaulay $\hat{R}^{+}$-algebra. Then $\tau_{B}(\hat{R},\Delta )$ is a monomial ideal and a monomial $x^{v} \in \tau_{B}(\hat{R},\Delta )$ if and only if
$$
v - w \in \Int(\sigma^{\vee}).
$$
If $\Gamma$ is an effective $\mathbb{Q}$-Cartier divisor with $n \Gamma = \divisor x^{w}$ then $x^{v} \in \tau_{B}(\omega_{\hat R}, \Gamma)$ if and only if
$$
v \in \Int(\Newt((x^{w}))).
$$
\end{theorem} 

\begin{proof}
First assume that $K_{\hat R} + \Delta$ is effective. Then by definition 
$$
\tau_{B}(\hat{R}, \Delta) = \Ann_{\omega_{\hat{R}}}\left(\ker\left( H^{d}_\mf{m}(\hat R) \xrightarrow{x^{w}}H^{d}_\mf{m}(B) \right)\right).
$$
By Lemma \ref{pure} the map $H^{d}_\mf{m}( \hat R) \xrightarrow{x^{w}}H^{d}_\mf{m}(B)$ can be factored through an intermediate extension $ \hat R_{n} = \hat{W[\sigma^{\vee} \cap \frac{1}{n}M}]$ where $n$ is chosen so that $u/r$ is an element of the fractional lattice $\frac{1}{n}M$. Let 
$$
K = \ker\left( H^{d}_\mf{m}( \hat R) \xrightarrow{x^{w}}H^{d}_\mf{m}( \hat R_{n}) \right).
$$
Using the explicit description of the local cohomology modules given in Lemma $\ref{lcm}$, everything is $\mathbb{Z}^{d-1}$ graded and so $K$ is monomial. Furthermore, it is clear that a monomial $x^{t}$ is in the kernel if and only if $t +w \not\in -\Int(\sigma^{\vee}) \cap \frac{1}{n}M$. 

To compute $\Ann_{\omega_{\hat R}}(K)$ it is necessary to understand the $\omega_{ \hat R}$ action on $H^{d}_{\mf{m}}(\hat R)$. This action is given by the perfect pairing
$$
H^{d}_{\mf{m}}(\hat R) \times \hat \omega_{R} \rightarrow \hat E
$$
where $\hat E$ is an injective hull of the residue field $\hat R /\mf{m}\hat{R}$ computed in Lemma \ref{inj}, and we are using the fact that $E(\hat R/m \hat R) \cong \hat{E(R/\mf{m})}$ and $\omega_{\hat R} \cong \hat \omega_{R}$. This pairing is induced by the isomorphism
$$
H^{d}_{\mf{m}}(\hat R) \cong \Hom_{\hat R}(\hat \omega_R, \hat E) \cong \hat{\Hom_{R}(\omega_{R}, E)} .
$$
We construct an explicit isomorphism  
\begin{align*}
H^{d}_{\mf{m}}( R) &\cong \Hom_{R}( \omega_R, E)  \\
x^{a} &\to \phi: x^{b} \rightarrow \left\{\begin{array}{ll} x^{a+b} & \mbox{ if } a+b \in -\sigma^{\vee} \\ 0 & \mbox{ otherwise} \end{array}\right.
\end{align*}
To see the above map is an isomorphism observe that a map from $\omega_{R}$ to $\hat E$ is an $\hat R$-module homomorphism if and only if there is ismorphism of the underlying lattices. Thus any $\phi \in \hat{\Hom_{R}(\omega_R, E)}$ is given by multiplication by a monomial. It follows that $\phi$ is nonzero if and only if $x^{a} \in -\Int(\sigma^{\vee})\cap M$.

To finish the proof note that 
$$
(x^{t} , x^{v}) = 0 \Leftrightarrow \langle t+v, v_{i} \rangle > 0 \mbox{ for some } i
$$
where $v_{i}$ are the first lattice points of the rays of $\sigma$. 

If $v -w \in \Int(\sigma^{\vee})$ for some $v \in \sigma^{\vee} \cap M$, then $\langle v,v_{i} \rangle > \langle w,v_{i} \rangle$ for all $i$. So for every $x^{t} \in K$ here exists some $i$ such that $\langle t,v_{i} \rangle \geq \langle t,w \rangle$. Then
$(x^{t}, x^{v}) = 0 \mbox{ for all } x^{t} \in K$ since for some $i$ $\langle t+v, v_{i} \rangle  = \langle t, v_{i} \rangle + \langle v, v_{i} \rangle > 0$. Thus if $v \in (\Int(\sigma^{\vee}) - w) \cap M$ then $x^{v} \Ann_{\hat  \omega_{R}}(K)$. 

For the reverse inclusion, note that if $v -w \not\in \Int(\sigma^{\vee})$ for $v \in \Int(\sigma^{\vee})\cap M$ then $x^{-v}\in K$ and $(x^{-v}, x^{v})$ does not pair to zero so $x^{v} \not\in \Ann_{\hat \omega_{R}}(K)$. 

If $\Delta + K_{\hat R}$ is not effective, we can find some $a \in \sigma^{\vee} \cap M$ such that $\Delta + K_{\hat R} + \divisor(x^{a})\geq 0$ . Then by \cite{MS2} Lemma 6.6
$$
\tau_{B}(\hat R,\Delta + \divisor(x^{a})) = x^{a} \cdot \tau_{B}(\hat R,\Delta).
$$
As before, let $\divisor(x^{u}) = r(K_{\hat R}+\Delta)$ and set $w = u/r$. Then we can use the previous part to compute the left hand side
$$
\tau_{B}(\hat R,\Delta + \divisor(x^{a})) = \{x^{v} \in \hat R \ | \ v - (w+a) \in \Int(\sigma^{\vee})  \}. 
$$
Then it is clear that $\tau_{B}(\hat R,\Delta)$ is monomial and a monomial $x^{v}$ is in $\tau_{B}(\hat R,\Delta)$ if and only if $v -w \in \Int(\sigma^{\vee})$.

The statement about $\tau_{B}(\omega_{\hat{R}}, \Gamma)$ follows immediately from the above since by hypothesis $w \in \sigma^{\vee}$ we can rewrite
$$
\tau_{B}(\hat{R}, \Gamma) = \{x^{v} \in \hat R \ | \ v - w \in \Int(\sigma^{\vee})\} = \{x^{v} \in  \hat R\ | \ v  \in \Int(\Newt(x^{w}))\}.
$$
\end{proof}

\subsection{BCM Test Ideals of Triples} We will use the same notation as in section 2.2.

\begin{defn} Let $(R, \mf{m})$ be a complete local domain, $\Delta$ as in Setting \ref{set} and $\mf{a}$ an ideal of $R$, $t \in \mathbb{R}_{+}$. By hypothesis there is some $g \in R$ such that $\divisor(g) = m(K_{X} + \Delta)$. Define
$$
0_{H^{d}_{\mf{m}}(R)}^{B,\Delta, \mf{a}^{t}} = \bigcap_{n=1}^{\infty} 
\ker \left( H^{d}_\mf{m}(R) \xrightarrow{g^{1/m}\mf{a}^{\ce{tn}/n}}H^{d}_\mf{m}(B) \right)
$$
Fix a sequence of generators $\{ f_{1}, \ldots, f_{n}\}$ for $\mf{a}$. We will abbreviate the set of generators as $[\underline{f}]$. Define
$$
0_{H^{d}_{\mf{m}}(R)}^{B,\Delta, [\underline{f}]^{t}} = \bigcap_{n=1}^{\infty} \bigcap_{f} \ker \left( H^{d}_\mf{m}(R) \xrightarrow{g^{1/m}f}H^{d}_\mf{m}(B) \right) 
$$
where $\divisor g = m(K_{X} + \Delta)$ and the second intersection ranges over all $f$ of the form
$$
f = \prod_{i=1}^{a} f_{j_{i}}^{1/n} \mbox{ where } a \geq tn \mbox{ for all } n \gg 0 .
$$
Then the big Cohen-Macaulay test ideal of the triple $((R,\Delta), [\underline{f}]^{t})$ with respect to $B$ is defined as
$$
 \tau_{B}((R,\Delta), [\underline{f}]^{t}) = \Ann_{\omega_{R}}(0_{H^{d}_{\mf{m}}(R)}^{B,\Delta, [\underline{f}]^{t}} ),
$$
and the big Cohen-Macaulay test ideal of the triple $((R,\Delta), \mf{a}^{t})$ with respect to $B$ is defined as
$$
 \tau_{B}((R,\Delta), \mf{a}^{t}) = \Ann_{\omega_{R}}(0_{H^{d}_{\mf{m}}(R)}^{B, \Delta,\mf{a}^{t}} ).
$$
\end{defn}

\begin{remark}
The definition given here differs slightly from the definition given in \cite{MS1}. In particular, the notation $[\underline{f}]$ used in \cite{MS1} requires not just a choice of generators but also a set of compatible $p$th roots of $f$. By working in an $R^{+}$-algebra we ensure that any choice of $f^{1/n}$ differs from another by a unit and hence does not impact the test ideal.  

Also note that the definition above omits an almost term $\omega^{1/p^{\infty}}$ and an $\epsilon$ perturbation of the exponent. As noted in \cite{MS2} it is expected that these two definitions should coincide for sufficiently large $B$. 
\end{remark}

A key difference between the two definitions is that in $\tau_{B}(\mf{a}^{t})$ we are taking $p^{e}$-th roots of sums of elements whereas in $\tau_{B}([f_{1}, \ldots, f_{n}]^{t}])$ we are not considering those roots. It follows that the two definitions may not define the same ideal but the following containment is straightforward to see.

\begin{proposition} \cite{MS2} \label{con} Fix $[\underline{f}] = \{f_{1}, \ldots, f_{n}\}$ a sequence of generators of an ideal $\mf{a} \subset R$. Then for all $t >0$ we have
$$
\tau_{B/\ms{B}}((R,\Delta), [\underline{f}]^{t}]) \subseteq \tau_{B/\ms{B}}((R, \Delta), \mf{a}^{t}).
$$
\end{proposition}

Another preliminary result we will need is the following theorem which compares the big Cohen-Macaulay test ideal and the multiplier ideal. 
\begin{theorem} \label{mult}
Let $R$ and $\Delta$ be as they are defined in Setting \ref{set}.  Then there exists an integral perfectoid big Cohen-Macaulay $R^{+}$-algebra $B$ such that: 
$$
\tau_{B}((R, \Delta), \mf{a}^{t}) \subseteq \mu_{*}\mathcal{O}_{Y}(K_{Y} - \floor{\mu^{*}(K_{X} + \Delta) + cA}).
$$
When $\pi$ is a resolution of singularities the right hand side is the multiplier ideal.
\end{theorem}

\begin{proof}
Fix a log resolution $\mu: Y \rightarrow X$ of $\mf{a}$. By \cite{MS2} Theorem 6.21 there exists an integral perfectoid big Cohen-Macaulay $R^{+}$-algebra $B$ such that
$$
\tau_{B}(R,\Delta) \subseteq \mu_{*}\mathcal{O}_{Y}(K_{Y} - \floor{\mu^{*}(K_{X} + \Delta)}).
$$
Summing over the appropriate divisors gives the desired result. 
\end{proof}

For a mixed characteristic toric ring $R = W[\sigma^{\vee} \cap M]$ we can construct a resolution of singularities using the same combinatorial process used in equal characteristic. See Appendix A for discussion.  A consequence is the following:

\begin{lemma} \label{cont}
Let $X =\Spec \hat{R} = \hat{\Spec W[ \sigma^{\vee} \cap M]}$ be the completion of a  mixed characteristic toric ring at the maximal ideal defined at the start of the section. Then for any monomial ideal $\mf{a}$ a monomial ideal not involving $p$, there exists an integral perfectoid big Cohen-Macaulay $\hat{R}^{+}$-algebra $B$ such that
$$
\tau_{B}((\hat{R}, \Delta), \mf{a}^{t}) \subseteq \{ x^{v} \in \hat{R} \ | \ v-w  \in\Int(t\Newt(\mf{a}) \}. 
$$
for all positive $t \in \mathbb{Q}$.
\end{lemma}

\begin{proof}
Construct $\mu: Y \rightarrow X$ a toric log resolution of $\mf{a}$ with $\mf{a} \cdot \mc{O}_{Y} = \mc{O}(-A)$ (see Appendix A). By Theorem \ref{mult} for any effective divisor $\Delta$ with $K_{X} + \Delta$ we have the following containment
$$
\tau_{B}((\hat R, \Delta), \mf{a}^{t}) \subseteq \pi_{*}\mc{O}_{Y}(  K_{Y} -\fl{\pi^{*}(K_{X} + \Delta) + tA}) = \mathcal{J}((\hat R,\Delta), \mf{a}^{t}).
$$
Observe the right hand side is the multiplier ideal because $\mu$ is a resolution of singularities. Completion commutes with the construction of the multiplier ideal, so it suffices to compute the multiplier ideal of $R$.

We make the following observations. Any valuation on $R_{p}:= R[1/p]$ naturally gives a valuation on $R$. Furthermore, if $\nu$ is a valuation on $R$ satisfying the property that $\nu(p) = 0$ then $\nu$ is a valuation on $R_{p}$. Since our resolution was constructed by blowing up monomials not involving $p$ any of the valuations needed to compute the multiplier ideal will satisfy $\nu(p) = 0$.

Setting
$$
K_{Y} -\fl{\pi^{*}(K_{X} + \Delta) + tA} = \sum a_{i} E_{i},
$$ 
we can express the multiplier ideal as
$$
\mathcal{J}((R,\Delta), \mf{a}^{t} ) = \{ f \in R \ | \ \mbox{ord}_{E_i}(\mu^{*}f) \geq a_{i} \mbox{ for all } i \} .
$$
By the above observation, each of the conditions comes from the corresponding valuation $\mbox{ord}_{E_i \otimes R_{p}}$ given by inverting $p$. Furthermore, there are no conditions that do not arise in this manner. Thus $\mathcal{J}((R,\Delta), \mf{a}^{t} )$ is a monomial ideal generated precisely by the monomials that generate $\mathcal{J}((R_{p},\Delta), \mf{a}^{t} )$ as desired. 
\end{proof}

\begin{theorem} \label{me}
Let $X = \Spec \hat R$ be the completion of a mixed characteristic toric ring at the maximal ideal defined at the start of the section. Let $(X, \Delta)$ be a pair with $\Delta$ a torus invariant $\mathbb{Q}$-divisor. Since $K_{X} + \Delta$ is $\mathbb{Q}$-Cartier and torus invariant, there exists a monomial $x^{u}$  such that div $x^{u}= r(K_{X} + \Delta)$ for some integer $r$. Setting $w = \frac{u}{r} $ we have that,
$$
\tau_{B}(( \hat R, \Delta), \mf{a}^{t}) = \{ x^{v} \in \hat  R \ | \ v-w  \in\Int(t\Newt(\mf{a}) \}. 
$$
\end{theorem}

\begin{proof} Fix a set of monomial generators $[\underline{f}=](f_{1}, \ldots, f_{k})$ for $\mf{a}$. Let 
$$
F_{n} = \{ f \ | \ f = \prod_{i=1}^{a} f_{j_i}^{1/n} \mbox{ with } a \geq tn\}.
$$
From the definition we can compute 
$$
\tau_{B}((\hat R,\Delta), [\underline{f}]^{t})  = \sum_{n=1}^{\infty} \sum_{\substack{\Gamma = \Delta + \divisor(f) \\ f \in F_{n}}} \tau_{B}(\hat R, \Gamma).
$$
Then by Theorem \ref{MRG} 
\begin{align*}
\tau_{B}((\hat R,\Delta), [\underline{f}]^{t})  &= \sum_{n=1}^{\infty} \sum_{f \in F_{n}} \{x^{v}\in \hat R\ | \ v -w\in \Int(\Newt(f))\\ &=  \{ x^{v} \in R \ | \ v-w  \in\Int(t\Newt(\mf{a})) \}. 
\end{align*}
Now the result follows from
$$
\tau_{B}((\hat R,\Delta), [\underline{f}]^{t}) \subseteq \tau_{B}(( \hat R,\Delta), \mf{a}^{t})\subseteq \mathcal{J}((\hat R,\Delta), \mf{a}^{t}) = \{ x^{v} \in \hat R \ | \ v-w  \in\Int(t\Newt(\mf{a})) \}, 
$$
where the first containment comes from Proposition \ref{con}, the second from Lemma \ref{cont} and the final equality from Proposition \ref{mult}.

\end{proof}
 
 \begin{example}
Here is an example to demonstrate the above theorem. Let $\sigma = \mbox{Cone}(e_{1} - e_{2}, e_{2})$, $W$ any mixed characteristic DVR, $R = W[\sigma^{\vee} \cap \mathbb{Z}^{2}]$ and $\mf{a} = (x^{5}y, x^{4}y^{3})$ a monomial ideal of $R$. Then for any sufficiently large big Cohen-Macaulay $\hat{R}^{+}$-algebra $B$ we can compute $\tau_{B}((\hat{R},\emptyset), \mf{a})$ as follows:

By the last part Theorem \ref{me} we are interested in computing
   $$
  \{ x^{v} \in \hat  R \ | \ v-w  \in\Int(\Newt(\mf{a}) \}, 
$$
 where $w = -(2,1)$ (since $\langle w,(1,-1) \rangle = -1$ and $\langle w, (0,1) \rangle = -1)$. 

The newton polyhedron $\Newt(\mf{a})$, pictured below inside of $\sigma^{\vee}$ as the region with the green boundary. 

\begin{figure}[H]
\begin{center}
  \begin{tikzpicture}[scale = .3]
    \coordinate (Origin)   at (0,0);
    \coordinate (XAxisMin) at (-.5,0);
    \coordinate (XAxisMax) at (15,0);
    \coordinate (YAxisMin) at (0,-.5);
    \coordinate (YAxisMax) at (0,15);
    \draw [ultra thick, black,-latex] (XAxisMin) -- (XAxisMax);
    \draw [ultra thick, black,-latex] (YAxisMin) -- (YAxisMax);

    \clip (-.5,-.5) rectangle (15cm,15cm); 
    \pgftransformcm{1}{0}{0}{1}{\pgfpoint{0cm}{0cm}}
    \coordinate (Bone) at (0,2);
    \coordinate (Btwo) at (2,-2);
    \foreach \x in {-7,-6,...,7}{
      \foreach \y in {-7,-6,...,7}{
        \node[draw,circle,inner sep=1pt,fill] at (2*\x,2*\y) {};
      }
    }

    \draw [thin,-latex,red, fill=gray, fill opacity=0.3] (0,0)
         -- ($2*(15,0)$)
         -- ($20*(1,1)$) -- cycle;
    \draw [thin,-latex,green, fill=black, fill opacity=0.3] (8,6) -- (10,2) -- ($(10,2) + 20*(1,0)$) -- ($(8,6) + 20*(1,1)$) -- cycle; 
    \filldraw [blue] (6,2) circle (2mm);
    \filldraw [blue] (6,4) circle (2mm);
    \draw [ultra thick, blue,-latex] (6,2) -- (10,4);
    \draw [ultra thick, blue,-latex] (6,4) -- (10,6);
  \end{tikzpicture}
  \caption{$\Newt(\mf{a})$ inside of $\sigma$}
  \end{center}
  \end{figure}
  From the picture we compute
  $$
  \tau_{B}(\hat{R}, \mf{a}) = (x^3y, x^3y^2).
  $$
  Note that monomials like $x^{3}$ and $x^{4}$ are not in $\tau_{B}$ because $(3,0) + (2,1)$ and $(4,0) + (2,1)$ are not in the interior of $\Newt(\mf{a})$.

\end{example}

 \appendix
\section{Mixed Characteristic Toric Resolutions}

In this section we review some additional toric geometry and then construct a toric resolution in mixed characteristic. In equal characteristic, as we will review below, we can construct a toric resolution of singularities by refining the fan. These refinements can be constructed by blowing up monomials. Our strategy for constructing a toric resolution in the mixed characteristic setting will be to blow up the same monomials as in the equal characteristic case and then arguing that this does indeed give a resolution of singularities. 

\begin{defn}
A fan $\Sigma$ in $M_{\mathbb{R}}$ is a finite collection of cones $\sigma \subseteq M_{\mathbb{R}}$ such that:
\begin{itemize}
\item Each $\sigma \in \Sigma$ is a strongly convex rational polyhedral cone.
\item For all $\sigma \in \Sigma$ every face of $\sigma$ is also in $\Sigma$.
\item For all $\sigma_{1}, \sigma_{2} \in \Sigma$, the intersection of $\sigma_{1} \cap \sigma_{2}$ is a face of each.
\end{itemize}
If $\Sigma$ is a fan then we define the support of $\Sigma$ denoted $|\Sigma|= \cup_{\sigma \in \Sigma}\sigma$. We will denote by $X(\Sigma)$ the mixed characteristic toric scheme associated to the fan $\Sigma$.
\end{defn}

\begin{defn}
Let $\Sigma$ be a fan in $M$ and $v$ a nonzero point in $|\Sigma|$. Define $\Sigma^{*}(v)$ to be the following set of cones:
\begin{itemize}
\item $\sigma$, such that $v \not \in \sigma \in \Sigma$.
\item Cone($\tau, v$), where $v \not\in \tau \in \Sigma$ and $\{v\} \cup \tau \subseteq \sigma \in \Sigma$.
\end{itemize}
We call $\Sigma^{*}(v)$ the star subdivision of $\Sigma$ at $v$. 
\end{defn}

\begin{proposition}[\cite{COX} Proposition 3.3.15]
With notation as above $\Sigma^{*}(v)$ is a refinement of $\Sigma$, and the induced toric morphism makes $X_{\Sigma^{*}(\sigma)}$ the blowup of $X_{\sigma}$ at the distinguished point $\gamma_{\sigma}$ of the cone $\sigma$. 
\end{proposition}

\begin{theorem} \label{res}  [\cite{COX} Theorem 11.1.9]
Every fan $\Sigma$ has a refinement $\Sigma'$ with the following properties
\begin{itemize}
\item $\Sigma'$ is smooth
\item $\Sigma'$ contains every smooth cone of $\Sigma$
\item $\Sigma'$ is obtained from $\Sigma$ by a sequence of star subdivisions
\item The toric morphism $X_{\Sigma'} \rightarrow X_{\Sigma}$ is a projective resolution of singularities.
\end{itemize}
\end{theorem}

\begin{theorem}
Let $X$ be an mixed characteristic toric scheme of the form $X = \Spec R = \Spec W[\sigma^{\vee} \cap M]$. There exists a resolution of singularities $\mu: Y \rightarrow X$ such that $Y[1/p]$ is a toric resolution of singularities of $X[1/p]$.
\end{theorem}

\begin{proof}
Let $X' =  \Spec R[1/p]$. Then by proposition \ref{res} there exists a toric resolution of singularities $\mu': Y' \rightarrow X'$ via a sequence of monomial blowups. Perform this same sequence of blowups on $X$ to obtain $Y$ with $\mu: Y \rightarrow X$. Since the sequence of blowups used to obtain $Y$ does not involve $p$ we have that $Y$ is precisely the mixed characteristic toric scheme defined by the combinatorial fan obtained in equal characteristic. We check can check that $Y$ is smooth locally, on any toric affine chart, near $p$, it is nonsingular because $\mathcal{O}_{Y/p}$ is, away from $P$ it is nonsingular because $\mathcal{O}_{Y[1/p]}$ is.  
\end{proof}

  \bibliographystyle{alpha}
 \bibliography{bib}

\end{document}